\documentclass[reqno]{amsart}
\usepackage[margin=3cm]{geometry}
\usepackage{amsmath}
\usepackage{amssymb}
\usepackage{amsbsy}
\usepackage[latin1]{inputenc}
\usepackage{dsfont}
\usepackage{empheq}

\usepackage{siunitx}
\usepackage{wasysym} 
\usepackage{amssymb}

\usepackage{enumerate}

\def\f        {{\boldsymbol f}}
\def\g        {{\boldsymbol g}}

\def\u        {{\boldsymbol u}}
\def\vv        {{\boldsymbol v}}
\def\w        {{\boldsymbol w}}
\def\x        {{\boldsymbol x}}

\def\B         {{\boldsymbol B}}

\def\LL       {{\boldsymbol L}}
\def\H        {{\boldsymbol H}}

\def\dx{{\rm d}{\boldsymbol x}}
\def\ds{{\rm d}s}
\def\dt{{\rm d}t}

\def\R {{\mathds R}}

\def\S {\boldsymbol {\mathcal{S}}}

\newtheorem{theorem}{Theorem}[section]
\newtheorem{lemma}{Lemma}[section]
\newtheorem{proposition}{Proposition}[section]

\newtheorem{definition}{Definition}[section]
\newtheorem{remark}{Remark}[section]

\begin{document}

\title[Potential singularities for Navier-Stokes solutions]{Two scenarios on a potential smoothness breakdown for the three-dimensional Navier-Stokes equations}

\author[J. V. Gutiérrez-Santacreu]{Juan Vicente Guti\'errez-Santacreu$^\ddag$}
\thanks{$\ddag$ Dpto. de Matemática Aplicada I, E. T. S. I. Informática, Universidad de Sevilla. Avda. Reina Mercedes, s/n. E-41012 Sevilla, Spain.  E-mail: {\tt juanvi@us.es}. This work was partially supported by the Spanish grant No. MTM2015-69875-P from Ministerio de Economía y Competitividad }

\date{\today}

\begin{abstract} 

In this paper we construct two families of initial data being arbitrarily large under any scaling-invariant norm for which their corresponding weak solution to the three-dimensional Navier-Stokes equations become smooth on either $[0,T_1]$ or $ [T_2,\infty)$, respectively, where $T_1$ and $T_2$ are two times prescribed previously. In particular, $T_1$ can be  arbitrarily large and $T_2$ can be arbitrarily small. Therefore, possible formation of singularities would occur after a very long or short evolution time, respectively.  

We further prove that if a large part of the kinetic energy is consumed prior to the first (possible) blow-up time, then the global-in-time smoothness of the solutions follows for the two families of initial data.

\end{abstract}
\maketitle


\section{Introduction}
The Cauchy problem of the Navier-Stokes equations for the flow of a viscous, incompressible, Newtonian fluid can be written as
\begin{equation}\label{PDE-NS}
\left\{
\begin{array}{rcccl}
\partial_t\vv-\Delta\vv+\nabla p+\vv\cdot\nabla\vv&=&\boldsymbol{0}&\mbox{ in }&\R^3\times (0, \infty),
\\
\nabla\cdot\vv&=&0&\mbox{ in }&\R^3\times (0,\infty).
\end{array}
\right.
\end{equation}
Here $\vv$ represents the velocity of the fluid and $p$ its pressure. It should be noted that the density and the viscosity have been normalized, as is always possible, by the rescaling argument on the time and space variable $\u(\frac{\nu t}{\rho}, \frac{\nu \x}{\rho})$ and $\frac{1}{\rho} p(\frac{\nu t}{\rho}, \frac{\nu\x}{\rho})$. 

To these equations we add an initial condition
\begin{equation}\label{IC}
\vv(0)=\vv_0\quad\mbox{ in }\quad\R^3,
\end{equation}
where $\vv_0 $ is a smooth, divergence-free vector field.

Despite considerable effort invested by scientific community, the mechanisms governing the solutions to the three-dimensional Navier-Stokes equations remain unsolved. At the present time, we do not know yet whether smooth solutions to the three-dimensional Navier-Stokes on $\R^3$ exist for all time. In other words, we do not know whether there are initially smooth solutions with finite energy of the Navier-Stokes equations that develop singularities in finite time.   

The mathematical existence theory developed so far supplies only partial answers to the smoothness of the Navier-Stokes equations. It is known that Navier-Stokes solutions are smooth on $[0,\infty)$ provided the initial velocity $\vv_0$ satisfies a smallness condition for certain norm. Instead, if the initial data $\vv_0$ are not assumed to be small, it is known that the time interval of existence is reduced to  $[0,T)$, where $T$ depends badly on some norm of $\vv_0$. 

\subsection{Previous results}

In 1934 Leray in his ground-breaking paper \cite{Leray} established the first result of local- and global-in-time existence of smooth solutions to the three-dimensional Navier-Stokes equations on $\R^3$. More precisely, Leray showed that there was a time interval $[0,T)$ for which $\LL^\infty(\R^3)$-norm solutions existed and hence they were smooth. He also proved that the Navier-Stokes equations had smooth solutions for all time under a smallness condition for $\|\vv_0\|_{\LL^2(\R^3)} \|\nabla\vv_0\|_{\LL^2(\R^3)}$ or $\|\vv_0\|^2_{\LL^2(\R^3)} \|\vv_0\|_{\LL^\infty(\R^3)}$.  Since that time, there has been quite a vast literature addressing local- and global-in-time existence results in different contexts. We will briefly discuss some works for critical spaces, which are those whose associated norm is invariant under the scaling $\lambda\u(\lambda \x, \lambda t^2)$ for all $\lambda>0$.

Fujita and Kato (1964) \cite{Fujita-Kato_1964} established the local- and global-in-time existence of $\overset{.}{\H}^{\frac{1}{2}}(\R^3)$-solutions. Twenty years later Kato \cite{Kato_1984} demonstrated that the three-dimensional Navier-Stokes equations are locally and globally well-posed in the $\LL^3(\R^3)$ space. The smoothness of $\LL^3(\R^3)$-solutions being Leray-Hopf weak solutions is due to Escauriaza, Seregin and Sverak (2003) \cite{Escauriaza-Seregin-Sverak_2003}.  

Afterwards came the work of Cannone (1995) \cite{Cannone_1995} in the Besov spaces $\overset{.}{\B}_{q,\infty}^{-1+3/q}(\R^3)$ for $q<\infty$. The next progress was the work of Koch and Tataru (2001) \cite{Koch_Tataru_2001} in the $BMO^{-1}(\R^3)$ space.  Solving the Navier-Stokes problem in $\overset{.}\B^{-1+3/q}_{q,\infty}(\R^3)$ or $BMO^{-1}(\R^3)$ allowed to construct  highly oscillating initial data $\vv_0$ with  $\|\vv_0\|_{\LL^3(\R^3)}$ being large as long as $\|\vv_0\|_{\overset{.}{\B}_{q,\infty}^{-1+3/q}(\R^3)}$ or $\|\vv_0\|_{BMO^{-1}(\R^3)}$ was small. Moreover, the smallness condition either on $\|\vv_0\|_{\overset{.}{\B}_{q,\infty}^{-1+3/q}(\R^3)}$ or $\|\vv_0\|_{BMO^{-1}(\R^3)}$ led to global $\LL^3(\R^3)$-solutions which combined with being Leray-Hopf solutions implied smoothness globally in time. Finally, we mention the work of Lei and Lin (2014) \cite{Lei_Lin_2011}  who proved the global-in-time well-posedness of solutions in the scaling invariant space
$$
\{\boldsymbol{f}\in\boldsymbol{\mathcal{D}}'(\R^3) : \int_{\R^3} |\boldsymbol{\xi}|^{-1} |\mathcal{F}\boldsymbol{f} (\boldsymbol{\xi})| {\rm d} \boldsymbol{\xi} \},
$$
where $\mathcal{F}$ stands for the Fourier transform. 

A turning point appeared with the result of Bourgain and Pavlovic (2008) \cite{Borgain_Palovic_2008} dealing with the Navier-Stokes problem in $\overset{.}{\B}_{\infty,\infty}^{-1}(\R^3)$. They showed that there were initial data in the Schwartz class $\boldsymbol{\mathcal{S}}(\R^3)$ being arbitrarily small in $\overset{.}{\B}_{\infty,\infty}^{-1}(\R^3)$ whose $\overset{.}{\B}_{\infty,\infty}^{-1}(\R^3)$-solutions become arbitrarily large after an arbitrarily short time. On the contrary, Chemin and Gallagher (2009) \cite{Chemin_Gallagher_2009} showed that there existed global $\overset{.}{\B}_{\infty,\infty}^{-1}(\R^3)$-solutions if a certain nonlinear smallness condition was satisfied. These two last results broke the pattern followed for scaling invariant spaces in the above indicated references -- Global-in-time well-posedness under a linear smallness condition for initial data. Even though Leray  \cite{Leray} already found nonlinear smallness conditions for proving the global-in-time existence of $\LL^\infty(\R^3)$-solutions. In this sense, Robinson and Sadowski (2014)\cite{Robinson_Sadowski_2014} have recently been published a result of local well-posedness under a smallness condition for  $\|\vv_0\|_{\LL^3(\R^3)}\int_0^T \int_{\R^3} |\nabla\u(s)|^2 |\u(s)| \ds$ where $\u(t)$ is the solution of the heat equation with the initial condition $\vv_0$.   

A change in the philosophy of constructing large initial data $\vv_0$ was to look for special structures which allowed to prove global-in-time existence. In this sense, Mahalov and Nicolaenko (2003) \cite{Makhalov_Nikolaenko_2003} constructed large initial data $\vv_0$ which transformed the Navier-Stokes equations into a rotating fluid equation. In such a setting, it is known that Navier-Stokes solutions are globally well-posed. Chemin and Gallagher (2009) \cite{Chemin_Gallagher_2009} proposed initial data which varied slowly in one direction. In these two examples the global well-posedness of two-dimensional Navier-Stokes equations is the crucial point in their proof.

Since our results rely on different ways of perturbing the Navier-Stokes equations for obtaining large solutions, we would like to mention some related works that study the concept of stability of solutions in certain spaces. Gallagher (2001) \cite{Gallagher_2001} proved that, for any sequence of initial data, their corresponding solution can be decomposed into a sum of orthogonal profiles bounded in $\overset{.}\H^\frac{1}{2}(\R^3)$ plus a remainder which is small with respect to the $\LL^3(\Omega)$-norm. As a result, the stability of solutions in $\overset{.}\H^\frac{1}{2}(\R^3)$ is proved for initial data in $\overset{.}\H^\frac{1}{2}(\R^2)\cap \LL^3(\R^3)$ being bounded in $\overset{.}H^\frac{1}{2}(\R^3)$ and providing $\LL^3(\R^3)$-solutions. The space $\LL^3(\R^3)$ could be changed by  $\overset{.}{\B}_{q,\infty}^{-1+3/q}(\R^3)$ or $BMO^{-1}(\R^3)$. This last result was extended, in \cite{Gallagher_Iftimie_Planchon_2003}, by Gallagher, Iftimie and Planchon (2003) to the stability of solutions in $\B^{-1+\frac{3}{p}}_{p,q}(\R^3)$ and  $\LL^3(\R^3)$.

%
%
%
%


\subsection{The contribution of this paper} Let us highlight the main contributions and how they differ from existing works concerning stability. 

In this paper we will construct smooth initial data $\vv_0$ being arbitrarily large in any critical space that do not develop singularities up to a given time $T_1$ without appealing to the two-dimensional Navier-Stokes equations. To achieve such a result we make use of Kato's technique. More precisely, the method of proof is based on mild-solution theory for proving the global-in-time existence of $\LL^3(\R^3)$-solutions for small data $\vv_0$. The main difference is that we do not directly impose a smallness condition on the $\LL^3(\R^3)$-norm for $\vv_0$. In doing so, we decompose the original problem into a Stokes problem with an initial datum $\u_0$ and a perturbed Navier-Stokes-like problem with an initial datum $\w_0$. From these two subproblems, we will prove that the three-dimensional Navier-Stokes problem possesses $\LL^3(\R^3)$-solutions with initial data $\vv_0=\u_0+\w_0$, where $\u_0$ has to be small concerning the $\LL^3(\R^3)$-norm and $\w_0$ has to be small concerning the $\LL^q(\R^3)$-norm. As a consequence, $\vv_0$ is no longer small in any critical space such as $\overset{.}\H^{\frac{1}{
2}}(\R^{2})$, $\LL^3(\R^3)$, $\overset{.}\B_{q,\infty}^{-1+\frac{3}{q}}(\R^3)$ or $BMO^{-1}(\R^3)$. This way we will rule out the smallness conditions for $\vv_0$. The result of Escauriaza,  Seregin, and {\v{S}}ver{\'a}k is the final ingredient to conclude with the construction of large initial data $\vv_0$ for the Navier-Stokes equations which provides smooth solutions on $[0,T_1]$, for $T_1$ being arbitrarily large. Consequently, the formation of potential singularities would have to be after $T_1$. This means that the system would preserve an enough amount of kinetic energy so that the solutions could blow up. On the other hand, if the $\LL^2(\R^2)$-value of the vorticity would keep large without blowing up so that the kinetic energy would decay under a certain threshold on $[0,T_1]$, the solutions starting from our initial data remained smooth for all time.

Moreover, if a different decomposition of (\ref{PDE-NS}) into a Navier-Stokes problem and a perturbed Navier-Stokes-like problem is used, we will be able to prove that there exist  Leray-Hopf weak solutions  becoming smooth on $[T_2,\infty)$ for any given time $T_2$. Then we infer that potential singularities would have to occur on $(0,T_2)$, for $T_2$ being arbitrarily small. The most kinetic energy would be consumed on $(0,T_2)$ so that the solutions can not experience new singularities on $[T_2,\infty)$.

In this paper we do not use the perturbation theory as a way of studying stability of solutions but a way of constructing large solutions to the Navier-Stokes equations. Particularly, if we used the stability theory developed  for some space $X$, with $X$ being  $\overset{.}\H^{\frac{1}{
2}}(\R^{2})$, $\LL^3(\R^3)$, $\overset{.}\B_{q,\infty}^{-1+\frac{3}{q}}(\R^3)$ or $BMO^{-1}(\R^3)$, we would obtain that there exists a number $\varepsilon$ (small enough) such that if $\|\vv_0-\u_0\|\le \varepsilon$ we have  
$$
\|\u(t)-\vv(t)\|_{X}\le E \|\vv_0-\u_0\|_{X}\quad\mbox{ for all }\quad t\in[0,T],
$$                            
where $\varepsilon>0$ and $E>0$ depend on some energy norms of the solution $\u(t)$. This would provide that the perturbed solution $\vv(t)$ would have an initial datum satisfying $\|\vv_0\|_{X}\le \varepsilon+\|\u_0\|_{X}$. But in order for the solution $\u(t)$ to exist on $[0,T]$ one requires some smallness condition for $\u_0$. Then, the solution $\vv(t)$ would inherit a smallness condition for $\vv_0$ and hence would not be large.

\section{Statement of Problem}
\subsection{Notation}
As usual, $\LL^p(\R^3)$, $1\le p\le +\infty$, denotes the space of $p$-integrable, Lebesgue-measurable, $\R^3$-valued functions defined on $\R^3$, and $\H^1(\R^3)$ denotes the space of functions $\vv \in \LL^2(\R^3 )$ such that $\nabla\vv \in \LL^2(\R^3 )$, where $\nabla$ is the gradient operator in the  distributional sense. Moreover, $\boldsymbol{C}_c(\R^3\times (0,T))$ is the space of infinitely continuously differentiable functions with compact supports in $\R^3\times (0,T)$. The Schwartz space is denoted as  $\boldsymbol{\mathcal{S}(\R^3)}$ representing the space of  rapidly decreasing infinitely continuously differentiable functions on $\R^3$.

For $X$ a Banach space,  $L^p(0,T; X)$ denotes the space of $p$-integrable, Bochner-measurable, $X$-valued functions on $(0,T)$. 

We let  $\mathcal{P}$ be the Helmholtz-Leray operator onto the space of divergence-free functions in $\LL^p(\R^3)$ with $1<p<\infty$.  
\subsection{The Navier-Stokes equations} 
 
In this paper the concept of \emph{weak solutions} for the Navier-Stokes problem~(\ref{PDE-NS})--(\ref{IC}) will be understood in the sense of Leray and Hopf (see \cite{Leray, Hopf}). 

\begin{definition}\label{def:weak-solution} A function $\vv(t)$ is said to be a Leray-Hopf weak solution of
problem (\ref{PDE-NS})--(\ref{IC}) if:
\begin{equation}\label{regularity}
\vv\in L^\infty(0,T; \LL^2(\R^3))\cap L^2(0,T; \H^1(\R^3))\quad \mbox{ with }\quad \nabla\cdot\vv=0,
\end{equation}
and
\begin{equation}\label{weak-formulation}
\begin{array}{ll}
\displaystyle -\int_0^T\int_{\R^3}\vv(s,\x)\cdot\partial_t\boldsymbol{\varphi}(s,\x)\dx\,\ds&\displaystyle+\int_0^T\int_{\R^3} \nabla\vv(s,\x):\nabla\boldsymbol{\varphi}(s,\x)\dx\,\ds
\\
&\displaystyle+\int_0^T
\int_{\R^3}\vv(s,\x)\cdot\nabla\vv(s,\x)\cdot\boldsymbol{\varphi}(s,\x)\dx\,\ds=(\u_0, \boldsymbol{\varphi}(0)),  
\end{array}
\end{equation}
for all $\boldsymbol{\varphi}\in\boldsymbol{C}^\infty_c (\R^3\times [0,T))$ with $\nabla\cdot \boldsymbol{\varphi}=0$. Moreover, the energy
inequality
\begin{equation}\label{Energy}
\frac{1}{2}\|\vv(t)\|^2_{\LL^2(\R^3)}+\int_{0}^{t}\|\nabla\vv(s)\|^2_{\LL^2(\R^3)}\,\ds\le
\frac{1}{2}\|\vv_0\|^2_{\LL^2(\R^3)}
\end{equation}
holds  a. e.~in $[0,T]$.
\end{definition}
Leray proved the global-in-time existence of weak solutions \cite{Leray}.  
\begin{theorem}\label{Leray} Let $\vv_0\in \LL^2(\R^3)$ be a divergence-free vector field. Then there exists at least a Leray-Hopf weak solution to (\ref{PDE-NS})--(\ref{IC}) on $[0, T]$.
\end{theorem}
Next we introduce the concept of strong (or regular) solutions to (\ref{PDE-NS})-(\ref{IC}).
\begin{definition}
A weak solution $\vv(t)$ to problem (\ref{PDE-NS})--(\ref{IC}) is said to be a strong solution if there exists a number $M_\vv>0$ such that
$$
\sup_{t\in[0, T] } \|\nabla\vv\|_{\LL^2(\R^3)}\le M_\vv.
$$  
\end{definition}
The key point for proving that solutions to the Navier-Stokes equations are smooth is to obtain that Leray-Hopf weak solutions are strong indeed, of course, for smooth initial data. 

Here we announce our two main results.

\begin{theorem}\label{Main} Let $T>1$ be given. Then there exist smooth, divergence-free initial data $\vv_0$ arbitrarily large under any critical norm such that their corresponding Leray-Hopf solution $\vv(t)$ to (\ref{PDE-NS})--(\ref{IC}) is smooth on $[0,T]$.
\end{theorem}

\begin{theorem}\label{Main2}
Let $0<T<1$ be given. Then there exist initial data $\vv_0$ arbitrarily large under any critical norm such that there exists at least a Leray-Hopf solution $\vv(t)$ to (\ref{PDE-NS})--(\ref{IC}) which is smooth on $[T,\infty)$.
\end{theorem}

Throughout this paper, different positive constants will appear due to interpolations and embeddings among spaces. Thus, $C$ will always be the maximum of all of these constants in the previous steps, and $K$ and $K'$ will stand for constants depending on the initial data.

\section{Proof of Theorem \ref{Main}}
In proving Theorem \ref{Main} we need to introduce a suitable approximation procedure so that all the estimates that follow are rigorously set up. To do this, we use a regularization \emph{à la Leray}. That is, we replace the nonlinearity $\vv\cdot\nabla \vv$ by $(\rho_\varepsilon*\vv)\cdot \nabla\vv$, where $\rho\in C_c^\infty(\R^3)$ such that $\rho\ge0$ and $\int_{\R^3} \rho(\x)\,\dx=1$ and $\rho_\varepsilon(\x)=\frac{1}{\varepsilon^2} \rho(\frac{\x}{\varepsilon})$ for all $\varepsilon>0$, to get 
\begin{equation}\label{NS-regularized}
\left\{
\begin{array}{rcccl}
\partial_t\vv_\varepsilon-\Delta\vv_\varepsilon+\nabla p_\varepsilon+(\rho_\varepsilon*\vv_\varepsilon)\cdot\nabla\vv_\varepsilon&=&\boldsymbol{0}&\mbox{ in }&\R^3\times (0, \infty),
\\
\nabla\cdot\vv_\varepsilon&=&0&\mbox{ in }&\R^3\times (0,\infty),
\end{array}
\right.
\end{equation}
associated with the regularized initial condition $\vv_\epsilon(0)=\vv_0$. This procedure gives rise to a solution pair $(\vv_\varepsilon, p_\varepsilon)\in \boldsymbol{C}^\infty(\R^3\times[0,\infty))\times C^{\infty}(\R^3\times[0,\infty))$.  On dealing with above equations, it is preferably better to avoid the pressure.  For this, we apply the Helmholtz-Leray  operator $\mathcal{P}$ to (\ref{NS-regularized}) to get  
\begin{equation}\label{NS}
\left\{
\begin{array}{rcl}
\displaystyle
\partial_t\vv_\varepsilon-\Delta\vv_\varepsilon + \mathcal{P}((\rho_\varepsilon*\vv_\varepsilon)\cdot\nabla\vv_\varepsilon)=\boldsymbol{0},
\\
\vv_\varepsilon(0)=\vv_0,
\end{array}
\right.
\end{equation}
where we have utilized the fact that $-\mathcal{P}\Delta\vv=-\Delta \mathcal{P}\vv=-\Delta\vv$ since $\mathcal{P}$ commutes with derivatives of any order.

From now on, for simplicity in exposition, we handle (\ref{NS}) without regularizing, although it must be taken into account in order to justify all the computations in this work. 

Our first step is to modify  equation (\ref{NS}) in order to easily produce a family of global smooth solutions. 
We first decompose (\ref{NS}) into two subproblems: a Stokes problem and a Navier-Stokes-like perturbation as follows.  Let $\u$ be the solution to the Stokes problem
\begin{equation}\label{Stokes}
\left\{
\begin{array}{rcl}
\partial_t\u- \Delta\u &=&\boldsymbol{0}, 
\\
\u(0)&=&\u_0,
\end{array}
\right.
\end{equation}
and let $\w$ be the solution to the perturbation problem 
\begin{equation}\label{Perturbation}
\left\{
\begin{array}{rcl}
\displaystyle
\partial_t\w - \Delta\w+\mathcal{P}(\u\cdot\nabla\w)+\mathcal{P}(\w\cdot\nabla\u)
+\mathcal{P}(\w\cdot\nabla\w)+\mathcal{P}(\u\cdot\nabla\u)&=&\boldsymbol{0},
\\
\w(0)&=&\w_0.
\end{array}
\right.
\end{equation}
Observe that defining $\vv=\u+\w$ and adding (\ref{Stokes}) and (\ref{Perturbation}), we  obtain (\ref{NS}) for $\vv_0=\u_0+\w_0$. In order to prove our main result, we need to write (\ref{Stokes}) and (\ref{Perturbation}), by using the Fourier transform, as 
\begin{equation}\label{Stokes_Heat}
\u(t)=K_t*\u_0
\end{equation}
and
\begin{equation}\label{Perturbation_Heat}
\w(t)=K_t*\w_0+\int_0^t K_{t-s} * (\mathcal{P}(\w\cdot\nabla\w)+\mathcal{P}(\u\cdot\nabla\w)+\mathcal{P}(\w\cdot\nabla\w)+\mathcal{P}(\u\cdot\nabla\u))  \ds,
\end{equation}  
where  $K_t=\frac{1}{(4\pi t)^\frac{3}{2}} e^{-\frac{|\x|^2}{4t}}$, for all $t>0$, is the heat kernel.

At this point we emphasize that, from (\ref{Stokes_Heat}) and (\ref{Perturbation_Heat}), we obtain the Duhamel integral form of (\ref{NS}):
\begin{equation}\label{Duhamel} 
\vv(t)=K_t*\vv_0+\int_0^t K_{t-s} * (\mathcal{P}(\vv\cdot\nabla\vv))  \ds,
\end{equation}
with $\vv_0=\u_0+\w_0$. The equivalence between equations (\ref{NS}) and (\ref{Duhamel}) and equations (\ref{Stokes_Heat}) and (\ref{Perturbation_Heat}) are ensured due to the regularity of $\vv$ or, more precisely, $\vv_\varepsilon$.

The following proposition is concerned with some properties of $K_t$. The proof is straightforward by using the properties of the convolution operator and the particular structure of $K_t$. 
\begin{proposition}\label{Pro1} It follows that, for all $1<p\le q <\infty$,
\begin{equation}\label{Inv_Lq-Lp}
\|K_t * \f\|_{\LL^q(\mathds{R}^3)}\le C t ^{-(\frac{1}{p}-\frac{1}{q})\frac{3}{2}} \|\f\|_{\LL^p(\mathds{R}^3)},
\end{equation}
\begin{equation}\label{Inv_Grad_Lq-Lp} 
\|\nabla K_t *\f \|_{\LL^q(\mathds{R}^3)}\le C t ^{-(1+\frac{3}{p}-\frac{3}{q})\frac{1}{2}} \|\f\|_{\LL^p(\mathds{R}^3)},
\end{equation}
where $C>0$ is a constant that does not depend on $\f$.  
\end{proposition}
\begin{proof} We will use the following property for the convolution operator:
$$
\|\f * \g\|_{\LL^q(\R^3)}\le \|\f\|_{\LL^r(\R^3)} \|\g\|_{\LL^p(\R^3)}\quad\mbox{ for }\quad\frac{1}{q}+\frac{1}{r}=1+\frac{1}{q}.
$$
Then we have
$$
\|K_t * \f\|_{\LL^q(\mathds{R}^3)}\le \|K_t\|_{\LL^r(\R^3)} \|\f\|_{\LL^p(\R^3)}.
$$
The proof of (\ref{Inv_Lq-Lp}) follows by observing that $\|K_t\|_{L^r(\R^3)}\le C t^{-(\frac{1}{p}-\frac{1}{q})}$. In the same way, we obtain that (\ref{Inv_Grad_Lq-Lp}) holds since $\|\nabla K_t\|_{\LL^r(\R^3)}\le C t ^{-(1+\frac{3}{p}-\frac{3}{q})\frac{1}{2}}$. 

\end{proof}

By Hölder's inequality and  Hodge's decomposition, we have
\begin{equation}\label{Interp_NL}
\|\mathcal{P}(\vv\cdot\nabla\vv)\|_{\LL^p(\mathds{R}^3)}\le C \|\vv\|_{\LL^r(\mathds{R}^3)} \|\nabla\vv\|_{\LL^s(\mathds{R}^3)}.
\end{equation}
for $\frac{1}{p}=\frac{1}{r}+\frac{1}{s}$.
  
From now on, we will assume $3<q$ and $\frac{1}{p}=\frac{1}{q}+\frac{1}{3}$ which implies that $3>p>\frac{3}{2}$. 

We will denote
$$
\beta(a, b)=\int_0^1 \gamma^{a-1} (1-\gamma)^{b-1}
$$ 
for all $a, b>0$. 

Next we provide some estimates for the solution to problem (\ref{Perturbation}) under a certain smallness condition for $\u_0$ and $\w_0$, respectively. 
\begin{lemma}\label{sec3-lm1} Let $T>1$ be given, and let  $\u_0\in \S (\R^3) $ and $\w_0\in \S(\R^3)$ be two divergence-free vector fields. Then there exists $K>0$ such that if
\begin{equation}\label{sec3-lm1-smallness_1}
T^{\frac{1}{2}} \max\{ \|\u_0\|_{\LL^q(\R^3)}, \|\nabla\u_0\|_{\LL^3(\R^3)} \}< \frac{K}{2^{\frac{1}{2}}4 } 
\end{equation}
and 
\begin{equation}\label{sec3-lm1-smallness_2}
 \|\w_0\|_{\LL^3(\mathds{R}^3)}< \frac{K}{2^{\frac{1}{2}} 4},
\end{equation} 
we have 
\begin{equation}\label{sec3-lm1_Lq}
t^{\frac{1}{2}(1-\frac{3}{q})} \|\w(t)\|_{\LL^q(\mathds{R}^3)}\le K
\end{equation}
and 
\begin{equation}\label{sec3-lm1_Grad_L3}
t^{\frac{1}{2}} \|\nabla\w(t)\|_{\LL^3(\mathds{R}^3)}\le K
\end{equation}
for all $t\in [0,T]$. 
\end{lemma}
\begin{proof} First of all, observe, from (\ref{Inv_Lq-Lp}) and (\ref{Inv_Grad_Lq-Lp}), that   
$$
\|K_t*\w_0\|_{\LL^q(\R^3)}\le C t^{-(1-\frac{3}{q})\frac{1}{2}}  \|\w_0\|_{\LL^3(\R^3)}
$$
and
$$
\|\nabla K_t*\w_0\|_{\LL^3(\R^3)}\le C t^{-\frac{1}{2}} \|\w_0\|_{\LL^3(\R^3)}. 
$$
Next, assume that (\ref{sec3-lm1_Lq}) and (\ref{sec3-lm1_Grad_L3}) hold. Then we will see that it requires that (\ref{sec3-lm1-smallness_1}) and (\ref{sec3-lm1-smallness_2}) are to be satisfied. Let us bound the right-hand side of (\ref{Perturbation}). We have, by (\ref{Inv_Lq-Lp}) and  (\ref{Interp_NL}), that  
$$
\begin{array}{rcl}
\displaystyle\|\int_{0}^t K_{t-s}*\mathcal{P}(\w\cdot\nabla\w) \ds\|_{\LL^q(\mathds{R}^3)}&\le&\displaystyle \int_{0}^t \|K_{t-s}*\mathcal{P}(\w\cdot\nabla\w) \|_{\LL^q(\mathds{R}^3)} \ds
\\
&\le &\displaystyle C \int_{0}^t (t-s)^{-(\frac{1}{p}-\frac{1}{q})\frac{3}{2}} \|\mathcal{P}(\w\cdot\nabla\w) \|_{\LL^p(\mathds{R}^3)} \ds
\\
&\le &\displaystyle C \int_{0}^t (t-s)^{-(\frac{1}{p}-\frac{1}{q})\frac{3}{2}} \|\w\|_{\LL^q(\mathds{R}^3)}\|\nabla\w\|_{\LL^3(\mathds{R}^3)} \ds
\\
&\le &\displaystyle C K^2 \int_{0}^t (t-s)^{-(\frac{1}{p}-\frac{1}{q})\frac{3}{2}} s^{ -\frac{1}{2}(1-\frac{3}{q})} s^{-\frac{1}{2}}  \ds
\\
&\le& C K^2 t^{ -\frac{1}{2}(1-\frac{3}{q})} \beta(\frac{3}{2 q}, \frac{1}{2})\le C K^2 t^{ -\frac{1}{2}(1-\frac{3}{q})},
\end{array}
$$
where we have utilized the change of variable $s=t\gamma$ to obtain $\beta(\frac{3}{2 q}, \frac{1}{2})$. Analogously, we obtain, from $\|\u(t)\|_{\LL^q(\R^3)}\le  \|\u_0\|_{\LL^q(\R^3)}$ and  $\|\nabla\u(t)\|_{\LL^3(\R^3)}\le  \|\nabla\u_0\|_{\LL^3(\R^3)}$ :
$$
\begin{array}{rcl}
\displaystyle 
\int_{0}^t \|K_{t-s}*\mathcal{P}(\u\cdot\nabla\w)\|_{\LL^q(\mathds{R}^3)} \ds&\le&\displaystyle \int_{0}^t \|K_{t-s}*\mathcal{P}(\u\cdot\nabla\w) \|_{\LL^q(\mathds{R}^3)} \ds
\\
&\le &\displaystyle C \int_{0}^t (t-s)^{-(\frac{1}{p}-\frac{1}{q})\frac{3}{2}} \|\mathcal{P}(\u\cdot\nabla\w) \|_{\LL^p(\mathds{R}^3)} \ds
\\
&\le &\displaystyle C \int_{0}^t (t-s)^{-(\frac{1}{p}-\frac{1}{q})\frac{3}{2}} \|\u\|_{\LL^q(\mathds{R}^3)}\|\nabla\w\|_{\LL^3(\mathds{R}^3)} \ds
\\
&\le &\displaystyle C K\|\u_0\|_{\LL^q(\R^3)} \int_{0}^t (t-s)^{-(\frac{1}{p}-\frac{1}{q})\frac{3}{2}} s^{-\frac{1}{2}}  \ds
\\
&\le& C K \|\u_0\|_{\LL^q(\R^3)} \beta(\frac{1}{2}, \frac{1}{2})
\\
&\le& C K \|\u_0\|_{\LL^q(\R^3)} T^{\frac{1}{2}(1-\frac{3}{q})}  t^{ -\frac{1}{2}(1-\frac{3}{q})},
\end{array}
$$
$$
\begin{array}{rcl}
\displaystyle\|\int_{0}^t K_{t-s}*\mathcal{P}(\w\cdot\nabla\u) \ds\|_{\LL^q(\mathds{R}^3)}&\le&\displaystyle \int_{0}^t \|K_{t-s}*\mathcal{P}(\w\cdot\nabla\u) \|_{\LL^q(\mathds{R}^3)} \ds
\\
&\le &\displaystyle C \int_{0}^t (t-s)^{-(\frac{1}{p}-\frac{1}{q})\frac{3}{2}} \|\mathcal{P}(\w\cdot\nabla\u) \|_{\LL^p(\mathds{R}^3)} \ds
\\
&\le &\displaystyle C \int_{0}^t (t-s)^{-(\frac{1}{p}-\frac{1}{q})\frac{3}{2}} \|\w\|_{\LL^q(\mathds{R}^3)}\|\nabla\u\|_{\LL^3(\mathds{R}^3)} \ds
\\
&\le &\displaystyle C K \|\nabla\u_0\|_{\LL^3(\R^3)} \int_{0}^t (t-s)^{-(\frac{1}{p}-\frac{1}{q})\frac{3}{2}} s^{-\frac{1}{2}(1-\frac{3}{q})}  \ds
\\
&\le&C K \|\nabla\u_0\|_{\LL^3(\R^3)} t^{\frac{1}{2}} t^{ -\frac{1}{2}(1-\frac{3}{q})}  \beta(\frac{1}{2}(1+\frac{3}{q}), \frac{1}{2})
\\
&\le&C K \|\nabla\u_0\|_{\LL^3(\R^3)} T^{\frac{1}{2}} t^{ -\frac{1}{2}(1-\frac{3}{q})}
\end{array}
$$
and 
$$
\begin{array}{rcl}
\displaystyle
\int_{0}^t \|K_{t-s}*\mathcal{P}(\u\cdot\nabla\u)\|_{\LL^q(\mathds{R}^3)} \ds&\le&\displaystyle C \|\u_0\|_{\LL^q(\mathds{R}^3)} \|\nabla\u_0\|_{\LL^3(\R^3)}  t^{\frac{1}{2}}\beta(1, \frac{1}{2})
\\
&\le&\displaystyle C \|\u_0\|_{\LL^q(\mathds{R}^3)} \|\nabla\u_0\|_{\LL^3(\R^3)} T^{1-\frac{3}{2q}}  t^{ -\frac{1}{2}(1-\frac{3}{q})}.
\end{array}
$$
Applying the above estimates to (\ref{Perturbation_Heat}),  we obtain
$$
\begin{array}{rcl}
t^{\frac{1}{2}(1-\frac{3}{q})} \|\w(t)\|_{\LL^q(\R^3)}&\le& C \|\w_0\|_{\LL^3(\R^3)}+ C K^2 +C T^{\frac{1}{2}(1-\frac{3}{q})} \|\u_0\|_{\LL^q(\R^3)}K
\\
&&+ C T^{\frac{1}{2}}  \|\nabla\u_0\|_{\LL^3(\R^3)} K + C T^{1-\frac{3}{2q}}  \|\u_0\|_{\LL^q(\R^3)} \|\nabla\u_0\|_{\LL^3(\R^3)}.
\end{array}
$$
Since $T>1$ and $q>3$, we also have:
$$
\begin{array}{rcl}
t^{\frac{1}{2}(1-\frac{3}{q})} \|\w(t)\|_{\LL^q(\R^3)}&\le& C \|\w_0\|_{\LL^3(\R^3)}+ C K^2 +C T^{\frac{1}{2}} \|\u_0\|_{\LL^q(\R^3)}K
\\
&&+ C T^{\frac{1}{2}}  \|\nabla\u_0\|_{\LL^3(\R^3)} K + C T  \|\u_0\|_{\LL^q(\R^3)} \|\nabla\u_0\|_{\LL^3(\R^3)}
\\
&\le& C \|\w_0\|_{\LL^3(\R^3)}+ C K^2 + 2 C T^{\frac{1}{2}} \max\{ \|\u_0\|_{\LL^q(\R^3)}, \|\nabla\u_0\|_{\LL^3(\R^3)} \}
 K
\\
&&+ C T  \max\{ \|\u_0\|_{\LL^q(\R^3)}, \|\nabla\u_0\|_{\LL^3(\R^3)}\}^2.
\end{array}
$$

Moreover, we have, by (\ref{Inv_Grad_Lq-Lp}) and (\ref{Interp_NL}), that  
$$
\begin{array}{rcl}
\displaystyle\int_{0}^t \|\nabla K_{t-s}*\mathcal{P}(\w\cdot\nabla\w)\|_{\LL^3(\mathds{R}^3)} \ds&\le&\displaystyle \int_{0}^t \|\nabla K_{t-s}*\mathcal{P}(\w\cdot\nabla\w) \|_{\LL^3(\mathds{R}^3)} \ds
\\
&\le &\displaystyle C \int_{0}^t (t-s)^{-\frac{3}{2p}} \|\mathcal{P}(\w\cdot\nabla\w) \|_{\LL^p(\mathds{R}^3)} \ds
\\
&\le &\displaystyle C \int_{0}^t (t-s)^{-\frac{3}{2p}} \|\w\|_{\LL^q(\mathds{R}^3)}\|\nabla\w\|_{\LL^3(\mathds{R}^3)} \ds
\\
&\le &\displaystyle C K^2 \int_{0}^t (t-s)^{-\frac{3}{2p}} s^{-\frac{1}{2}(1-\frac{3}{q})} s^{-\frac{1}{2}}  \ds
\\
&\le& C K^2 t^{ -\frac{1}{2}}  \beta(\frac{1}{2}(1-\frac{3}{q}), \frac{3}{2q}) \le C K^2 t^{ -\frac{1}{2}}.
\end{array}
$$
Analogously, 
$$
\begin{array}{rcl}
\displaystyle
\int_{0}^t \|\nabla K_{t-s}*\mathcal{P}(\u\cdot\nabla\w)\|_{\LL^3(\mathds{R}^3)}\ds &\le&\displaystyle \int_{0}^t \|\nabla K_{t-s}*\mathcal{P}(\u\cdot\nabla\w) \|_{\LL^3(\mathds{R}^3)} \ds
\\
&\le &\displaystyle C \int_{0}^t (t-s)^{-\frac{3}{2p}} \|\mathcal{P}(\u\cdot\nabla\w) \|_{\LL^p(\mathds{R}^3)} \ds
\\
&\le &\displaystyle C \int_{0}^t (t-s)^{-\frac{3}{2p}} \|\u\|_{\LL^q(\mathds{R}^3)}\|\nabla\w\|_{\LL^3(\mathds{R}^3)} \ds
\\
&\le &\displaystyle C K\|\u_0\|_{\LL^q(\R^3)} \int_{0}^t (t-s)^{-\frac{3}{2p}}  s^{-\frac{1}{2}}  \ds
\\
&\le& C K \|\u_0\|_{\LL^q(\R^3)}  t^{ \frac{1}{2}(1-\frac{3}{p})}  \beta(\frac{1}{2},1-\frac{3}{2p} ) 
\\
&\le& C K \|\u_0\|_{\LL^q(\R^3)} T^{1-\frac{3}{2p}} t^{ -\frac{1}{2}},
\end{array}
$$
$$
\begin{array}{rcl}
\displaystyle
\int_{0}^t \|\nabla K_{t-s}*\mathcal{P}(\w\cdot\nabla\u) \|_{\LL^3(\mathds{R}^3)} \ds&\le&\displaystyle \int_{0}^t \|\nabla K_{t-s}*\mathcal{P}(\w\cdot\nabla\u) \|_{\LL^3(\mathds{R}^3)} \ds
\\
&\le &\displaystyle C \int_{0}^t (t-s)^{-\frac{3}{2p}} \|\mathcal{P}(\w\cdot\nabla\u) \|_{\LL^p(\mathds{R}^3)} \ds
\\
&\le &\displaystyle C \int_{0}^t (t-s)^{-\frac{3}{2p}} \|\w\|_{\LL^q(\mathds{R}^3)}\|\nabla\u\|_{\LL^3(\mathds{R}^3)} \ds
\\
&\le &\displaystyle C K\|\nabla\u_0\|_{\LL^3(\R^3)} \int_{0}^t (t-s)^{-\frac{3}{2p}}  s^{-\frac{1}{2}(1-\frac{3}{q})} \ds
\\
&\le& C K \|\nabla\u_0\|_{\LL^3(\R^3)} \beta(\frac{1}{2}+\frac{3}{2p}, 1-\frac{3}{2p}) 
\\
&\le& C K \|\nabla\u_0\|_{\LL^3(\R^3)} T^{\frac{1}{2}}  t^{ -\frac{1}{2}}
\end{array}
$$
and 
$$
\int_{0}^t \|\nabla K_{t-s}*\mathcal{P}(\u\cdot\nabla\u)\|_{\LL^3(\mathds{R}^3)} \ds\le C \|\u_0\|_{\LL^q(\R^3)} \|\nabla\u_0\|_{\LL^3(\R^3)}  t^{-\frac{1}{2}} T^{\frac{3}{2}(1-\frac{1}{p})} .
$$
Applying the above estimates to (\ref{Perturbation_Heat}), we obtain
$$
\begin{array}{rcl}
t^{\frac{1}{2}}\|\nabla\w(t)\|_{\LL^3(\R^3)}&\le& C \|\w_0\|_{\LL^3(\R^3)}+ C K^2 +C T^{1-\frac{3}{2p}} \|\u_0\|_{\LL^q(\R^3)}K
\\
&&+ C T^\frac{1}{2}  \|\nabla\u_0\|_{\LL^3(\R^3)} K + C T^{\frac{3}{2}(1-\frac{1}{p})}   \|\u_0\|_{\LL^q(\R^3)} \|\nabla\u_0\|_{\LL^3(\R^3)}.
\end{array}
$$
From the relation  $\frac{1}{p}=\frac{1}{q}+\frac{1}{3}$, we write
$$
\begin{array}{rcl}
t^{\frac{1}{2}}\|\nabla\w(t)\|_{\LL^3(\R^3)}&\le& C \|\w_0\|_{\LL^3(\R^3)}+ C K^2 +C T^{\frac{1}{2}(1-\frac{3}{q})} \|\u_0\|_{\LL^q(\R^3)}K
\\
&&+ C T^\frac{1}{2}  \|\nabla\u_0\|_{\LL^3(\R^3)} K + C T^{1-\frac{3}{2q}}   \|\u_0\|_{\LL^q(\R^3)} \|\nabla\u_0\|_{\LL^3(\R^3)}
\end{array}
$$
and hence 
$$
\begin{array}{rcl}
t^{\frac{1}{2}}\|\nabla\w(t)\|_{\LL^3(\R^3)}&\le& C \|\w_0\|_{\LL^3(\R^3)}+ C K^2 +C T^{\frac{1}{2}}\|\u_0\|_{\LL^q(\R^3)}K
\\
&&+ C T^\frac{1}{2}  \|\nabla\u_0\|_{\LL^3(\R^3)} K + C T  \|\u_0\|_{\LL^q(\R^3)} \|\nabla\u_0\|_{\LL^3(\R^3)}.
\\
&\le& C \|\w_0\|_{\LL^3(\R^3)}+ C K^2 + 2 C T^{\frac{1}{2}} \max\{ \|\u_0\|_{\LL^q(\R^3)}, \|\nabla\u_0\|_{\LL^3(\R^3)} \}
 K
\\
&&+ C T  \max\{ \|\u_0\|_{\LL^q(\R^3)}, \|\nabla\u_0\|_{\LL^3(\R^3)}\}^2.
\end{array}
$$


To close the bootstrap argument, we need to find  $K>0$ such that 
$$
\begin{array}{rcl}
K&=& C \|\w_0\|_{\LL^3(\R^3)}+ C K^2 + 2 C T^{\frac{1}{2}} \max\{ \|\u_0\|_{\LL^q(\R^3)}, \|\nabla\u_0\|_{\LL^3(\R^3)} \}
 K
\\
&&+ C T  \max\{ \|\u_0\|_{\LL^q(\R^3)}, \|\nabla\u_0\|_{\LL^3(\R^3)}\}^2.

\end{array}
$$
or, equivalently, 
$$
\begin{array}{rcl}
0&= &C \|\w_0\|_{\LL^3(\R^3)}+ C K^2+ (2CT^{\frac{1}{2}} \max\{\|\u_0\|_{\LL^q(\R^3)},\|\nabla\u_0\|_{\LL^3(\R^3)}\}
-1) K 
\\
&&+ C T  \max\{\|\u_0\|_{\LL^q(\R^3)}, \|\nabla\u_0\|_{\LL^3(\R^3)}\}^2.
\end{array}
$$
Let us choose $ 2CT^{\frac{1}{2}} \max\{\|\u_0\|_{\LL^q(\R^3)}, \|\nabla\u_0\|_{\LL^3(\R^3)}\}
-1<-\frac{1}{2}$.  As a result, we obtain that $K$ satisfies 
$$
0<K\le \frac{\frac{1}{2}-\sqrt{\frac{1}{4}-4C^2 (\|\w_0\|_{\LL^3(\mathds{R}^3)}+ T  
\max\{\|\u_0\|_{\LL^q(\R^3)}, \|\nabla\u_0\|_{\LL^3(\R^3)}\}^2} )}{2C}.
$$
Let us choose  $\frac{1}{4}-4C^2 (\|\w_0\|_{\LL^3(\mathds{R}^3)}+ T  \max\{\|\u_0\|_{\LL^q(\R^3)}, \|\nabla\u_0\|_{\LL^3(\R^3)}\}^2)>0$. From conditions (\ref{sec3-lm1-smallness_1}) and  (\ref{sec3-lm1-smallness_2})  for $\max\{\|\u_0\|_{\LL^q(\R^3)}, \|\nabla\u_0\|_{\LL^3(\R^3)}\}$ and $\|\w_0\|_{\LL^3(\R^3)}$, the two above conditions hold. Thus, (\ref{sec3-lm1_Lq}) and (\ref{sec3-lm1_Grad_L3}) are also satisfied. It completes the proof.   
\end{proof}

As a consequence of Lemma \ref{sec3-lm1}, we will infer an estimate uniform in time for the $\|\cdot\|_{\LL^3(\R^3)}$-norm of $\w$ on [0,T].   

\begin{lemma}\label{sec3-lm2} Let $\u_0\in \boldsymbol{\mathcal S}(\R^3)$ and $\w_0\in \boldsymbol{\mathcal S}(\mathds{R}^3)$ be two divergence-free vector fields  satisfying (\ref{sec3-lm1-smallness_1}) and (\ref{sec3-lm1-smallness_2}). Then there exists a number $M_\w>0$ such that the solution $\w(t)$ to (\ref{Perturbation}) satisfies
$$
\sup_{t\in[0, T]}\|\w(t)\|_{\LL^3(\mathds{R}^3)}\le M_\w.
$$
\end{lemma}
\begin{proof} From (\ref{Perturbation_Heat}), we have 
$$
\|\w(t)\|_{\LL^3(\mathds{R}^3)}\le \|K_t*\w_0\|_{\LL^3(\R^3)}+\int_0^t \|K_{t-s} * (\mathcal{P}(\w\cdot\nabla\w)+\mathcal{P}(\u\cdot\nabla\w)+\mathcal{P}(\w\cdot\nabla\w)+\mathcal{P}(\u\cdot\nabla\u))\|_{\LL^3(\mathds{R}^3)}  \ds.
$$
Let us now bound each term on the right-hand side. We have, by (\ref{Inv_Lq-Lp}), (\ref{sec3-lm1-smallness_1}), and (\ref{sec3-lm1-smallness_2}), that  
$$
\|K_t*\w_0\|_{\LL^3(\R^3)}\le C \|\w_0\|_{\LL^3(\R^3)}
$$
$$
\begin{array}{rcl}
\displaystyle
\int_0^t\|K_{t-s} * \mathcal{P}(\w\cdot\nabla\w)\|_{\LL^3(\mathds{R}^3)}\ds&\le&\displaystyle C\int_0^t (t-s)^{-(\frac{3}{p}-1)\frac{1}{2}} \|\w\cdot\nabla\w\|_{\LL^p(\mathds{R}^3)}\ds
\\
&\le& C\displaystyle \int_0^t (t-s)^{-\frac{3}{2q}} \|\w\|_{\LL^q(\mathds{R}^3)} \|\nabla\w\|_{\LL^{3}(\mathds{R}^3)} \ds
\\
&\le & C K^2\displaystyle \int_0^t (t-s)^{-\frac{3}{2q}} s^{-(1-\frac{3}{q})\frac{1}{2}}  s^{-\frac{1}{2}} \ds
\\
&\le& C K^2\beta(\frac{1}{2}(2-\frac{3}{q}), \frac{3}{q})\le C K^2,
\end{array}
$$
$$
\begin{array}{rcl}
\displaystyle
\int_0^t\|K_{t-s} * \mathcal{P}(\u\cdot\nabla\w)\|_{\LL^3(\mathds{R}^3)}\ds&\le&\displaystyle C\int_0^t (t-s)^{-(\frac{3}{p}-1)\frac{1}{2}} \|\u\cdot\nabla\w\|_{\LL^p(\mathds{R}^3)}\ds
\\
&\le& C\displaystyle \int_0^t (t-s)^{-\frac{3}{2q}} \|\u\|_{\LL^q(\mathds{R}^3)} \|\nabla\w\|_{\LL^{3}(\mathds{R}^3)} \ds
\\
&\le & C K\displaystyle \int_0^t (t-s)^{-\frac{3}{2q}} s^{-(1-\frac{3}{q})\frac{1}{2}} \|\u\|_{\LL^3(\R^3)} s^{-\frac{1}{2}} \ds
\\
&\le& C K  \|\u_0\|_{\LL^3(\mathds{R}^3)}  \beta(\frac{1}{2}(2-\frac{3}{q}), \frac{3}{q})\le C K \|\u_0\|_{\LL^3(\mathds{R}^3)} ,
\end{array}
$$
$$
\begin{array}{rcl}
\displaystyle
\int_0^t\|K_{t-s} * \mathcal{P}(\w\cdot\nabla\u)\|_{\LL^3(\mathds{R}^3)}\ds&\le&\displaystyle C\int_0^t (t-s)^{-(\frac{3}{p}-1)\frac{1}{2}} \|\w\cdot\nabla\u\|_{\LL^p(\mathds{R}^3)}\ds
\\
&\le& C\displaystyle \int_0^t (t-s)^{-\frac{3}{2q}} \|\w\|_{\LL^q(\mathds{R}^3)} \|\nabla\u\|_{\LL^{3}(\mathds{R}^3)} \ds
\\
&\le & C K\displaystyle \int_0^t (t-s)^{-\frac{3}{2q}} s^{-\frac{1}{2}(1-\frac{3}{q})} s^{-\frac{1}{2}} \|\u\|_{\LL^3(\mathds{R}^3)}  \ds
\\
&\le& C K \|\u_0\|_{\LL^3(\mathds{R}^3)} \beta(\frac{1}{2}(2-\frac{3}{q}), \frac{3}{q})\le  C K \|\u_0\|_{\LL^3(\mathds{R}^3)}
\end{array}
$$
and
$$
\begin{array}{rcl}
\displaystyle
\int_0^t\|K_{t-s} * \mathcal{P}(\u\cdot\nabla\u)\|_{\LL^3(\mathds{R}^3)}\ds&\le&\displaystyle C\int_0^t (t-s)^{-(\frac{3}{p}-1)\frac{1}{2}} \|\u\cdot\nabla\u\|_{\LL^p(\mathds{R}^3)}\ds
\\
&\le& C\displaystyle \int_0^t (t-s)^{-\frac{3}{2q}} \|\u\|_{\LL^q(\mathds{R}^3)} \|\nabla\u\|_{\LL^{3}(\mathds{R}^3)} \ds
\\
&\le & C \displaystyle \int_0^t (t-s)^{-\frac{3}{2q}} s^{-(\frac{1}{3}-\frac{1}{q})\frac{3}{2}} s^{-\frac{1}{2}}  \|\u\|_{\LL^3(\mathds{R}^3)}  \ds
\\
&\le& C\|\u_0\|^2_{\LL^3(\mathds{R}^3)}   \beta(\frac{1}{2}(2-\frac{3}{q}), \frac{3}{q})\le C\|\u_0\|^2_{\LL^3(\mathds{R}^3)}.
\end{array}
$$
Therefore, we obtain
\begin{equation}\label{sec3-lm2-lab1}
\begin{array}{rcl}
\|\w(t)\|_{\LL^3(\mathds{R}^3)}&\le& C \|\w_0\|_{\LL^3(\R^3)}+ C K^2+2 C \|\u_0\|_{\LL^3(\mathds{R}^3)} K +
C  \|\u_0\|_{\LL^3(\mathds{R}^3)}^2:=M_\w.
\end{array}
\end{equation}
\end{proof}
In view of Lemmas \ref{sec3-lm1} and \ref{sec3-lm2}, we have proved the existence of an $\LL^3(\R^3)$-solution to (\ref{NS}) on $[0,T]$ under certain smallness conditions for $\u_0$ and $\w_0$. 

\begin{lemma}\label{sec3-lm3}  Let $\u_0\in \S(\R^3)$ and $\w_0\in \S(\mathds{R}^3)$ be two divergence-free vector fields  satisfying (\ref{sec3-lm1-smallness_1}) and (\ref{sec3-lm1-smallness_2}), respectively. Then there exists $M_\vv>0$ such that the solution $\vv(t)$ to (\ref{NS}) with  $\vv_0=\u_0+\w_0$ satisfies
$$
\sup_{t\in [0,T]}\|\vv(t)\|_{\LL^3(\mathds{R}^3)}\le M_\vv,
$$
\end{lemma}
\begin{proof} First notice that, from (\ref{Stokes_Heat}), we have $\|\u(t)\|_{\LL^3(\mathds{R}^3)}\le \|\u_0\|_{\LL^3(\mathds{R}^3)}:=M_\u$ for $t\in[0,T]$. By Lemma \ref{sec3-lm2}, we have $\|\w(t)\|_{\LL^3(\mathds{R}^3)}\le M_\w$ for all $t\in[0, T]$. Therefore, if we define $\vv(t)=\u(t)+\w(t)$, we obtain $\|\vv(t)\|_{\LL^3(\mathds{R}^3)}\le M_\u+M_\w:=M_\vv$ for all $t\in[0,T]$, where $\vv$ satisfies (\ref{NS}) on $[0,T]$ with $\vv_0=\u_0+\w_0$.  
\end{proof} 

\begin{remark} It is not hard to see that the above estimate obtained for the regularized solutions are independent of $\varepsilon$ and hence are also true for the solutions of the unregularized Navier-Stokes equations. From now on, we are allowed to work  without any regularization procedure.      
\end{remark}

Our next goal is to provide a family of smooth initial data $\vv_0$ which can be split into $\u_0$ and $\w_0$ satisfying  (\ref{sec3-lm1-smallness_1}) and (\ref{sec3-lm1-smallness_2}), respectively. In doing so, we take advantage of the ``scissors effect " of the scaling property of the Navier-Stokes solutions. That is, we will use different scalings for $\u_0$ and $\w_0$ so that  supercritical and subcritical norms increase and decrease oppositely with the $\LL^3(\R^3)$-norm being invariant. This way we avoid that the size of any norm of $\vv_0$ is no longer small but large.          

Let $\boldsymbol{\vartheta}_0\in\S(\R^3)$ be a divergence-free vector field. We are allowed to take $\varepsilon >0$ such that $\boldsymbol{\vartheta}_0=(1-\varepsilon) \boldsymbol{\vartheta}_0+ \varepsilon\boldsymbol{\vartheta}_0:=\u_{0,\varepsilon}+\w_{0,\varepsilon}$ so that $\w_{0,\varepsilon}$ satisfies condition (\ref{sec3-lm1-smallness_2}). Next we define $\u_{0,\varepsilon}^{\tilde\lambda}=\tilde\lambda\u_{0,\varepsilon}(\tilde\lambda \x)$ and $\w_{0,\varepsilon}^{\hat\lambda}=\hat\lambda\w_{0,\varepsilon}(\hat\lambda \x)$ for $\tilde\lambda, \hat\lambda>0$. Letting $\tilde \lambda$ go to $0$, we find that there exists $\tilde\lambda_0$ such that, for all $\tilde\lambda\le \tilde\lambda_0$, it follows that condition (\ref{sec3-lm1-smallness_1}) holds for $\u_{0,\varepsilon}^{\tilde\lambda}$. Moreover, for any $\hat \lambda$, we find that $\w_{0,\varepsilon}^{\tilde\lambda}$ fulfills condition (\ref{sec3-lm1-smallness_2}) since the $\LL^3(\R^3)$-norm is scaling invariant. This last rescaling is not really necessary, but it allows us to construct initial data arbitrarily large under any supercritical norm. Thus we define $\vv_0=\u_{0,\varepsilon}^{\tilde\lambda}+\w_{0,\varepsilon}^{\hat\lambda}$ whose $\LL^3(\R^3)$-norm remains almost invariant due to our special choice, i.e. $\|\vv_0\|_{\LL^3(\R^3)}\le \|\u_0\|_{\LL^3(\R^3)}+\|\w_0\|_{\LL^3(\R^3)}\le (1-\varepsilon) \|\boldsymbol{\vartheta}_0\|_{\LL^3(\R^3)}+\varepsilon\|\boldsymbol{\vartheta}_0\|_{\LL^3(\R^3)}=\|\boldsymbol{\vartheta}_0\|_{\LL^3(\R^3)}$. Instead, supercritical norms can be arbitrarily large by doing $\hat\lambda$ to tend to $\infty$ and subcritical norms by doing $\tilde\lambda$ to tend to $0$.

Another possibility to construct smooth initial data $\vv_0$ is as follows. Consider $\tilde \u_0\in\S(\R^3)$ and $\tilde \w_0\in\S(\R^3)$ to be two divergence-free vector fields and define $\vv_0= \tilde\u^{\tilde\lambda}_0+\varepsilon\tilde\w_0^\lambda:=\u^{\hat\lambda}_0 +\w_{0,\varepsilon}$. Pick $\tilde\lambda$ to be such that $\u_0^{\tilde\lambda}$ satisfies condition (\ref{sec3-lm1-smallness_1}) and $\varepsilon$ to be such that $\w_{0,\varepsilon}$ satisfies condition (\ref{sec3-lm1-smallness_2}). 

The following theorem was proved in \cite{Escauriaza-Seregin-Sverak_2003} by Escauriaza, Seregin, and {\v{S}}ver{\'a}k. 
\begin{theorem}\label{ESS} Let $\vv_0\in\boldsymbol{\mathcal{S}}(\mathds{R}^3)$ be a divergence-free vector field.  Assume that $\vv(t) $ is a weak Leray-Hopf solution to (\ref{PDE-NS})--(\ref{IC}) and satisfies the additional condition 
$$
\sup_{t\in[0, T]}\|\vv(t)\|_{\LL^3(\mathds{R}^3)}<\infty. 
$$
Then $\vv(t)$ is a strong solution to (\ref{PDE-NS})--(\ref{IC}) on $[0,T]$. 
\end{theorem}
Therefore, Lemma \ref{sec3-lm3} and Theorem \ref{ESS} combined with Theorem \ref{Leray} give that the solutions $\vv(t)$ whose initial data $\vv_0$ can be decomposed as, for instance, $\vv_0=\u_{0,\varepsilon}^{\tilde\lambda}+\w_{0,\varepsilon}^{\hat\lambda}$ with $\u_{0,\varepsilon}^{\tilde\lambda}\in\boldsymbol{\mathcal{S}}(\R^3)$ and $\w_{0,\varepsilon}^{\hat\lambda}\in\boldsymbol{\mathcal{S}}(\R^3)$ being divergence-free vector fields fulfilling (\ref{sec3-lm1-smallness_1}) (for certain $\tilde\lambda$) and (\ref{sec3-lm1-smallness_2}) (for certain $\varepsilon$) are strong, and hence they are smooth on $[0,T]$. It proves Theorem \ref{Main}.

\begin{remark} It is easy to see that the solutions given in Theorem \ref{Main} satisfy the estimate:
$$
\|\vv(t)\|_{\LL^q(\R^3)}\le (K+ C \|\u_0\|_{\LL^3(\R^3)}) t^{-\frac{1}{2}(1-\frac{3}{q})}\quad\mbox{ for all } t\in[0,T]. 
$$
This implies that $\|\vv(T)\|_{\LL^q(\R^3)}$ can be as small as required provided that $T$ is large. As a result, we can extend our solution to $[0, T^*)$ for $T^*$ being possible large. See \cite[Thm 15.3]{Lamerie_2002}.
\end{remark}

\section{Proof of Theorem \ref{Main2}}
We first decompose (\ref{NS}) as follows. Let $\w_\varepsilon$ be the solution to the Navier-Stokes problem
\begin{equation}\label{NS_II}
\left\{
\begin{array}{rcl}
\partial_t\w_\varepsilon- \Delta\w_\varepsilon+{\mathcal P}((\rho_\varepsilon*\w_\varepsilon)\cdot\nabla\w_\varepsilon) &=&\boldsymbol{0}, 
\\
\w_\varepsilon(0)&=&\w_0,
\end{array}
\right.
\end{equation}
and let $\u_\varepsilon$ be the solution to the perturbation problem 
\begin{equation}\label{Perturbation_II}
\left\{
\begin{array}{rcl}
\displaystyle
\partial_t\u_\varepsilon - \Delta\u_\varepsilon+\mathcal{P}((\rho_\varepsilon*\w_\varepsilon)\cdot\nabla\u_\varepsilon)+\mathcal{P}((\rho_\varepsilon*\u_\varepsilon)\cdot\nabla\w_\varepsilon)
+\mathcal{P}((\rho_\varepsilon*\u_\varepsilon)\cdot\nabla\u_\varepsilon)&=&\boldsymbol{0},
\\
\u_\varepsilon(0)&=&\u_0.
\end{array}
\right.
\end{equation}
As before, we drop the subscript $\varepsilon$ and the convolution operator from (\ref{NS_II}) and (\ref{Perturbation_II}).

The following result is a consequence of Lemmas \ref{sec3-lm1} and \ref{sec3-lm2}. In particular, we assume that  we have $C>1$ in Lemma \ref{sec3-lm1}.   
  
\begin{lemma}\label{sec4-lm1} Let $\w_0\in \S(\R^3)$ be a divergence-free vector field such that 
\begin{equation}\label{sec4-lm1-smallness_I}
\|\w_0\|_{\LL^3(\R^3)}\le\frac{1}{4C}.
\end{equation} 
Then there exists a smooth solution $\w(t)$ to (\ref{NS_II}) on $[0,\infty)$ such that 
\begin{equation}\label{sec4-lm1-L3-bound}
\sup_{t\in[0,\infty)} \|\w(t)\|_{\LL^3(\R^3)}< K:=\frac{1-\sqrt{1-4C^2 \|\w_0\|_{\LL^3(\mathds{R}^3)}}}{2C}. 
\end{equation}
\end{lemma} 
\begin{proof} From (\ref{Stokes}) for $\u_0=\boldsymbol{0}$ and (\ref{Perturbation}) , we recover (\ref{NS_II}). By following the proof of Lemma \ref{sec3-lm1}, we obtain that (\ref{sec3-lm1_Lq}) and (\ref{sec3-lm1_Grad_L3}) hold for $K$ such that
$$
0=C\|\w_0\|_{\LL^3(\R^3)}+ C K^2-K.
$$
Therefore,
$$
K=\frac{1-\sqrt{1-4C^2 \|\w_0\|_{\LL^3(\mathds{R}^3)}}}{2C}.
$$
In virtue of (\ref{sec3-lm2-lab1}), we obtain (\ref{sec4-lm1-L3-bound}).    
\end{proof}


\begin{lemma}\label{sec4-lm2}
Let $0<T<1$ be given. Let  $\u_0\in \S (\R^3) $ and $\w_0\in \S(\R^3)$ be two divergence-free vector fields such that 
\begin{equation}\label{sec4-lm2-smallness_I}
\max_{t\in[0,T]}\|\w(t)\|_{\LL^3(\R^3)}< \frac{1}{8 C} 
\end{equation}
and 
\begin{equation}\label{sec4-lm2-smallness_II}
 T^{-\frac{1}{4}} \|\u_0\|_{\LL^2(\R^3)}< \frac{1}{8 C }.
\end{equation} 
Then there exists $t^*\in(0,T]$ such that 
\begin{equation}\label{sec4-lm2-t-star-in-L3}
 \|\u(t^*)\|_{\LL^3(\mathds{R}^3)}< \frac{1}{8 C}.
\end{equation} 
\end{lemma}
\begin{proof}

Multiplying (\ref{Perturbation_II}) by $\u$ and integrating over $\R^3$ gives, after integration by parts,
$$
\begin{array}{rcl}
\displaystyle
\frac{1}{2}\frac{\rm d}{\dt}\|\u\|^2_{\LL^2(\R^3)}+\|\nabla\u\|^2_{\LL^2(\R^3)}&=&
\displaystyle
\int_{\R^3}\u\cdot\nabla\w\cdot\u\,\dx=-\int_{\R^3}\u\cdot\nabla\u\cdot\w\,\dx
\\
&\le&\|\u\|_{\LL^6(\R^3)} \|\nabla\u\|_{\LL^2(\R^3)} \|\w\|_{\LL^3(\R^3)}
\\
&\le& C \|\w\|_{\LL^3(\R^3)} \|\nabla\u\|^2_{\LL^2(\R^3)}. 
\end{array}
$$ 
From (\ref{sec4-lm2-smallness_I}), we arrive at 
$$
\frac{\rm d}{\dt}\|\u\|^2_{\LL^2(\R^3)}+\|\nabla\u\|^2_{\LL^2(\R^3)}\le 0. 
$$
Integrating with respect to time, we get
$$
\sup_{t\in[0,T]}\|\u(t)\|^2_{\LL^2(\R^3)}+\int_0^T\|\nabla\u(s)\|^2_{\LL^2(\R^3)}\ds\le \|\u_0\|^2_{\LL^2(\R^3)},
$$
whence
$$
\sup_{t\in[0,T]}\|\u(t)\|^2_{\LL^2(\R^3)}+ 2 \int_0^T\|\nabla\u(s)\|^2_{\LL^2(\R^3)}\ds\le C \|\u_0\|^2_{\LL^2(\R^3)}.  
$$

By interpolation, we write 
$$
\|\u(t)\|_{\LL^3(\R^3)}\le C \|\u(t)\|_{\LL^2(\R^3)}^\frac{1}{2} \|\nabla\u(t)\|^\frac{1}{2}_{\LL^2(\R^3)}.
$$
Therefore,
$$
\int_0^T\|\u(t)\|^4_{\LL^3(\R^3)}\dt\le C \|\u_0\|^4_{\LL^2(\R^3)} 
$$
and hence
$$
T \inf_{s\in[0,T]} \|\u\|^4_{\LL^3(\R^3)}\le C \|\u_0\|^4_{\LL^2(\R^3)}
$$
and
$$
 \inf_{s\in[0,T]} \|\u\|^2_{\LL^3(\R^3)}\le C T^{-\frac{1}{2}} \|\u_0\|^2_{\LL^2(\R^3)}
.  
$$
If conditions (\ref{sec4-lm2-smallness_I}) and (\ref{sec4-lm2-smallness_II}) hold, there exists $t^*\in(0,T]$ such that condition (\ref{sec4-lm2-t-star-in-L3})  is satisfied.
\end{proof}

In order for condition (\ref{sec4-lm2-smallness_I}) to hold, we need
\begin{equation}\label{sec4-smallness-w0}
1-\sqrt{1-4C^2 \|\w_0\|_{\LL^3(\mathds{R}^3)}}< \frac{1}{4},
\end{equation}
which holds from (\ref{sec4-lm1-L3-bound}).
Let us choose $\tilde\lambda$ and $\varepsilon$ such that $\vv_0=\u^{\tilde\lambda}_{0,\varepsilon}+\w_{0,\varepsilon}^{\hat\lambda}$ with $\u_0$ and $\w_0$ satisfying  (\ref{sec4-lm2-smallness_II})  and (\ref{sec4-smallness-w0}), respectively. Thus,  we arrive at  
$$
\|\vv(t^*)\|_{\LL^3(\R^3)}\le\|\w(t^*)\|_{\LL^3(\R^3)}+\|\u(t^*)\|_{\LL^3(\R^3)}< \frac{1}{4C}.
$$
Then, by Lemma \ref{sec4-lm1}, we obtain 
$$
\sup_{t\in[T,\infty)}\|\vv(t)\|_{\LL^3(\R^3)}\le \frac{1}{2C}.
$$
since $\vv(t)$ is a solution of the regularized Navier-Stokes equations as $\w$.

As a result of Theorem \ref{ESS}, we have accomplished to prove that the unregularized solutions $\vv(t)$ whose initial data $\vv_0$ can be decomposed as $\vv_0=\u_{0,\varepsilon}^{\tilde\lambda}+\w_{0,\varepsilon}^{\hat\lambda}$ with $\u_{0,\varepsilon}^{\tilde\lambda}\in\boldsymbol{\mathcal{S}}(\R^3)$ and $\w_{0,\varepsilon}^{\hat\lambda}\in\boldsymbol{\mathcal{S}}(\R^3)$ being divergence-free vector fields fulfilling (\ref{sec4-smallness-w0}) and (\ref{sec4-lm2-smallness_II}), respectively,  are strong, and hence they are smooth on $[T,\infty)$. We have used the same decomposition for $\vv_0$ as in the proof of Theorem \ref{Main}. This way our initial conditions are arbitrarily large under any critical norm. It proves Theorem \ref{Main2}.

\section{Additional results}

To complete the proof of Theorem \ref{Main} we show there exist initial data $\vv_0$ which can not be  {\it a priori} decomposed as above. To do this, we just need to use, for instance, an $\LL^3(\R^3)$-stability result. The proof combines ideas from \cite{Escauriaza-Seregin-Sverak_2003} for establishing local-in-time existence of $\LL^3(\R^3)$-solutions and from \cite{Gallagher_Iftimie_Planchon_2003} for proving stability in Bevov spaces.   
\begin{theorem} Let $\vv(t)$ be a smooth solution to (\ref{PDE-NS})--(\ref{IC}) with an initial datum $\vv_0=\u_0+\w_0$, where $\u_0$ and $\w_0$ are two smooth, divergence-free vector fields fulfilling  (\ref{sec3-lm1-smallness_1}) and (\ref{sec3-lm1-smallness_2}), respectively. Then there exists $\varepsilon=\varepsilon(\vv)$ such that, for all initial data $\tilde\vv_0$ with  $\|\vv_0-\tilde\vv_0\|_{\LL^3(\R^3)}<\varepsilon$, the corresponding solution $\tilde\vv(t)$ with $\tilde \vv(0)=\tilde\vv_0$ satisfies
\begin{equation}\label{sec5-th1-L3-stability}
\|\vv(t)-\tilde\vv(t)\|_{\LL^3(\R^3)}\le C(\vv) \|\vv_0-\tilde\vv_0\|_{\LL^3(\R^3)} \mbox{ for all }t\in[0,T].
\end{equation}
\end{theorem}
\begin{proof} The proof is divided into two parts:

{\bf Part I}: A  priori estimates 

To start with, define $\w(t):=\vv(t)-\tilde\vv(t)$ to be the solution to 
\begin{equation}\label{sec5-th1-lab1}
\left\{
\begin{array}{rcl}
\displaystyle
\partial_t\w - \Delta\w+\nabla q+ \vv\cdot\nabla\w+\w\cdot\nabla\vv
+\w\cdot\nabla\w&=&\boldsymbol{0},
\\
\w(0)&=&\w_0:=\vv_0-\tilde\vv_0.
\end{array}
\right.
\end{equation}
Testing by $|\w|\w$, we obtain
$$
\begin{array}{rcl}
\displaystyle
\frac{1}{3}\frac{d}{dt} \|\w\|^3_{\LL^3(\R^3)}+ \||\w|^{\frac{1}{2}}\nabla \w\|_{\LL^2(\R^3)}^2+\frac{4}{9} \|\nabla|\w|^{\frac{3}{2}}\|_{\LL^2(\R^3)}^2&=&\displaystyle\int_{\R^3} \nabla q \cdot |\w|\w\,  \dx
\\
&&\displaystyle -\int_{\R^3} \nabla\cdot(\vv\w+\w\vv+\w\w)\cdot |\w|\w\,  \dx
\end{array}
$$
Integrating by parts, we estimate each term on the right hand side as follows. For the pressure term, applying  the divergence operator to (\ref{sec5-th1-lab1}), we first observe that 
\begin{equation}\label{sec5-th1-lab2}
-\Delta q=\nabla\cdot \nabla\cdot(\u\w+\u\w+\w\w)\quad\mbox{ in }\quad \R^3.
\end{equation} 
The Calderon-Zygmund inequality applied to (\ref{sec5-th1-lab2}) 
 (see \cite[Lm 5.1]{Robinson_Sadowski_2014} for a proof) implies that 
$$
\|q\|_{\LL^\frac{5}{2} (\R^3)}\le C \|\w\|_{\LL^5(\R^3)} (\|\w\|_{\LL^5(\R^3)}+\|\vv\|_{\LL^5(\R^3)}). 
$$
Thus, 
$$
\begin{array}{rcl}
\displaystyle
\int_{\R^3} \nabla q \cdot |\w|\w  \dx&=&\displaystyle-\int_{\R^3} q \nabla\cdot (\w|\w|)\dx=\int_{\R^3} q \nabla\cdot \w |\w|\dx+\int_{\R^3} q \w\nabla\w \frac{\w}{|\w|}\dx 
\\
&\le& \displaystyle \|q\|^\frac{1}{2}_{\LL^{\frac{5}{2}}(\R^3)} \|\w\|^{\frac{1}{2}}_{\LL^5(\R^3)} \||\w|^{\frac{1}{2}}\nabla \w\|_{\LL^2(\R^3)}.
\end{array}
$$
Next the interpolation inequality $\|\cdot\|_{\LL^{\frac{10}{3}}(\R^3)}\le C \|\cdot\|_{\LL^2(\R^3)}^\frac{2}{5} \|\nabla\cdot\|_{\LL^2(\R^3)}^\frac{3}{5}$ leads to
\begin{equation}\label{sec5-th1-lab3}
\|\w\|_{\LL^5(\R^3)}=\||\w|^{\frac{3}{2}}\|_{\LL^\frac{10}{3}(\R^3)}^{\frac{2}{3}}\le C \|\w\|_{\LL^3(\R^3)}^{\frac{2}{5}} \|\nabla|\w|^{\frac{3}{2}}\|_{\LL^2(\R^3)}^{\frac{2}{5}}.             
\end{equation}
From (\ref{sec5-th1-lab3}) and Young's inequality, we arrive at 
$$
\begin{array}{rcl}
\displaystyle
\int_{\R^3} \nabla q \cdot |\w|\w  \dx&\le& C \|\w\|_{\LL^3(\R^3)}^\frac{1}{2} \|\w\|_{\LL^5(\R^3)}^\frac{5}{2} (\|\w\|^\frac{5}{2}_{\LL^5(\R^3)}+\|\vv\|_{\LL^5(\R^5)}^\frac{5}{2})
\\
&&+  \gamma \||\w|^{\frac{1}{2}}\nabla \w\|_{\LL^2(\R^3)}^2+\delta \|\nabla|\w|^{\frac{3}{2}}\|^2_{\LL^2(\R^3)}.
\end{array}
$$
The other term for the pressure term is also bounded as: 
$$
\begin{array}{rcl}
\displaystyle
\int_{\R^3} q \w\nabla\w \frac{\w}{|\w|}\dx&\le& C \|\w\|_{\LL^3(\R^3)}^\frac{1}{2} \|\w\|_{\LL^5(\R^3)}^\frac{5}{2} (\|\w\|^\frac{5}{2}_{\LL^5(\R^3)}+\|\vv\|_{\LL^5(\R^5)}^\frac{5}{2})
\\
&&+  \gamma \||\w|^{\frac{1}{2}}\nabla \w\|^2_{\LL^2(\R^3)}+\delta \|\nabla|\w|^{\frac{3}{2}}\|^2_{\LL^2(\R^3)}.
\end{array}
$$
In the same way, we bound the remainder terms:
$$
\begin{array}{rcl}
\displaystyle
\int_{\R^3} \nabla\cdot(\vv\w)\cdot |\w|\w\,\dx&=&\displaystyle-\int_{\R^3} \vv\w \nabla \w |\w|\,\dx-\int_{\R^3} \vv \w \nabla\w  \frac{\w}{|\w|}\,d\x 
\\
&\le& C \|\w\|_{\LL^3(\R^3)}^\frac{1}{2} \|\w\|_{\LL^5(\R^3)}^\frac{5}{2} \|\vv\|_{\LL^5(\R^3)}^\frac{5}{2}
\\
&&+\gamma \||\w|^{\frac{1}{2}}\nabla \w\|^2_{\LL^2(\R^3)}+\delta \|\nabla|\w|^{\frac{3}{2}}\|^2_{\LL^2(\R^3)} ,
\end{array}
$$
$$
\begin{array}{rcl}
\displaystyle
\int_{\R^3} \nabla\cdot(\w\vv)\cdot |\w|\w\,\dx&\le& C \|\w\|_{\LL^3(\R^3)}^\frac{1}{2} \|\w\|_{\LL^5(\R^3)}^\frac{5}{2} \|\vv\|_{\LL^5(\R^3)}^\frac{5}{2}
\\
&&+\gamma \||\w|^{\frac{1}{2}}\nabla \w\|^2_{\LL^2(\R^3)}+\delta \|\nabla|\w|^{\frac{3}{2}}\|^2_{\LL^2(\R^3)} ,
\end{array}
$$
and 
$$
\begin{array}{rcl}
\displaystyle
\int_{\R^3} \nabla\cdot(\w\w)\cdot |\w|\w\,\dx&\le& C \|\w\|_{\LL^3(\R^3)}^\frac{1}{2} \|\w\|^5_{\LL^5(\R^3)}
\\
&&+\gamma \||\w|^{\frac{1}{2}}\nabla \w\|^2_{\LL^2(\R^3)}+\delta \|\nabla|\w|^{\frac{3}{2}}\|^2_{\LL^2(\R^3)} .
\end{array}
$$
Adjusting $\gamma$ and $\delta$ adequately, we find that 
$$
\frac{1}{3}\frac{d}{dt} \|\w\|^3_{\LL^3(\R^3)}+ \frac{1}{3} \||\w|^{\frac{1}{2}}\nabla \w\|_{\LL^2(\R^3)}^2+\frac{2}{9} \|\nabla|\w|^{\frac{3}{2}}\|_{\LL^2(\R^3)}\le C \|\w\|_{\LL^3(\R^3)}^\frac{1}{2} (\|\w\|^5_{\LL^5(\R^3)}+ \|\w\|_{\LL^5(\R^3)}^\frac{5}{2} \|\vv\|_{\LL^5(\R^3)}^\frac{5}{2}).
$$
Integrating over $(T_i,T_{i+1})$, where $\{T_i\}_{i=1}^M$ are to be determined later on,  yields 
$$
\begin{array}{rcl}
\|\w(t)\|_{\LL^3(\R^3)}^3&\le&\displaystyle \|\w(T_i)\|^3_{\LL^3(\R^3)}+ C \int_{T_i}^{T_{i+1}} \|\w\|_{\LL^3(\R^3)}^\frac{1}{2} (\|\w\|^5_{\LL^5(\R^3)}+ \|\w\|_{\LL^5(\R^3)}^\frac{5}{2} \|\vv\|_{\LL^5(\R^3)}^\frac{5}{2})  \ds
\\
&\le&\displaystyle \|\w(T_i)\|^3_{\LL^3(\R^3)}+ C \|\w(t)\|_{L^\infty({T_i}, T_{i+1};\LL^3(\R^3)}^\frac{1}{2} \int_{T_i}^{T_{i+1}} (\|\w\|^5_{\LL^5(\R^3)}+ \|\w\|_{\LL^5(\R^3)}^\frac{5}{2} \|\vv\|_{\LL^5(\R^3)}^\frac{5}{2})  \ds
\\
&\le&\displaystyle \|\w(T_i)\|^3_{\LL^3(\R^3)}+ \frac{1}{2} \|\w(t)\|_{L^\infty(T_i, T_{i+1};\LL^3(\R^3))}^3 + C \|\w\|_{L^5(T_i,T_{i+1};\LL^5(\R^3))}^6
\\
&&+ C \|\w\|_{L^5(T_i,T_{i+1};\LL^5(\R^3))}^3\|\vv\|_{L^5(T_i,T_{i+1};\LL^5(\R^3))}^3.
\end{array}
$$
In particular, this shows
\begin{equation}\label{sec5-th1-lab4}
\begin{array}{rcl}
\|\w\|_{L^\infty(T_i,T_{i+1};\LL^3(\R^3))}&\le& C \|\w(T_i)\|_{\LL^3(\R^3)}+ C \|\w\|_{L^5(T_i,T_{i+1};\LL^5(\R^3))}^2
\\
&&+ C\|\w\|_{L^5(T_i,T_{i+1};\LL^5(\R^3))}\|\vv\|_{L^5(T_i,T_{i+1};\LL^5(\R^3))},
\end{array}
\end{equation}
which implies that 
\begin{equation}\label{sec5-th1-lab5}
\begin{array}{rcl}
\displaystyle
\int_{T_i}^{T_{i+1}}(\frac{1}{2} \||\w|^{\frac{1}{2}}\nabla \w\|_{\LL^2(\R^3)}^2+\frac{2}{9} \|\nabla|\w|^{\frac{3}{2}}\|^2_{\LL^2(\R^3)}) \ds&\le& C\|\w(T_i)\|_{\LL^3(\R^3)}+C \|\w\|_{L^5(T_i,T_{i+1};\LL^5(\R^3))}^2
\\
&&+ C\|\w\|_{L^5(T_i,T_{i+1};\LL^5(\R^3))}\|\vv\|_{L^5(T_i,T_{i+1};\LL^5(\R^3))}.
\end{array}
\end{equation}
We now use (\ref{sec5-th1-lab3}) together with (\ref{sec5-th1-lab5})  to get  
\begin{equation}\label{sec5-th1-lab6}
\begin{array}{rcl}
\|\w\|_{L^5(T_i, T_{i+1}; \LL^5(\R^5))}&\le& C \|\w\|_{L^\infty (T_i, T_{i+1}; \LL^3(\R^3))}^{\frac{3}{5}} \|\nabla|\w|^\frac{3}{2}\|^\frac{2}{5}_{L^2(T_i,T_{i+1; \LL^2(\R^3))}}
\\
&\le&  C \|\w(T_i)\|_{\LL^3(\R^3)}+ C \|\w\|_{L^5(T_i,T_{i+1};\LL^5(\R^3))}^2
\\
&&+ C \|\w\|_{L^5(T_i,T_{i+1};\LL^5(\R^3))}\|\vv\|_{L^5(T_i,T_{i+1};\LL^5(\R^3))}.                             
\end{array}
\end{equation}

{\bf Part II}: Induction argument.

Since $\vv\in L^5(0,T; \LL^5(\R^3))$, there exists a finite sequence $\{T_i\}_{i=0}^M$ such that $[0,T]=\cup_{i=0}^{M-1} [T_i, T_{i+1}]$ satisfying 
\begin{equation}\label{sec5-th1-lab7}
\|\vv\|_{L^5(T_i, T_{i+1}; \LL^5(\R^3)}<\frac{1}{4 C}.
\end{equation}
where $C>0$ is the constant appearing in (\ref{sec5-th1-lab6}).

Let us consider 
\begin{equation}\label{sec5-th1-lab8}
\|\w_0\|_{\LL^3(\R^3)}\le \frac{1}{8C (2 C )^M}. 
\end{equation}
Then we claim that 
\begin{equation}\label{sec5-th1-lab9}
\|\w\|_{L^5(T_i, T_{i+1}; \LL^5(\R^5))}\le (2 C)^{i+1}\|\w_0\|_{\LL^3(\R^3)}
\end{equation}
and 
\begin{equation}\label{sec5-th1-lab10}
\|\w\|_{L^\infty(T_i, T_{i+1}; \LL^3(\R^3))}\le (2 C)^{i+1}\|\w_0\|_{\LL^3(\R^3)},
\end{equation}
for all $i\in\{0, \cdots, M-1\}$.

For $i=0$, let us suppose that there exists $K'>0$ such that 
$$
\|\w\|_{L^5(0,T_1;\LL^5(\R^3))}\le K'.   
$$
Then, from (\ref{sec5-th1-lab6}),  (\ref{sec5-th1-lab7}) and  (\ref{sec5-th1-lab8}), we find that  
$$
\|\w\|_{L^5(0,T_1;\LL^5(\R^3))}\le C \|\w_0\|_{\LL^3(\R^3)} +C (K')^2+\frac{1}{2} \|\w\|_{L^5(0,T_1;\LL^5(\R^3))}.
$$
We now impose that $K'$ satisfies 
$$
K'= 	C \|\w_0\|_{\LL^3(\R^3)} +C (K')^2+\frac{1}{2} \|\w\|_{L^5(0,T_1;\LL^5(\R^3))},
$$
which gives
$$
0<K'= \frac{\frac{1}{2}-\sqrt{\frac{1}{4}-4C^2 \|\w_0\|_{\LL^3(\mathds{R}^3)}}}{2C}.
$$
This implies the existence of $\w$ on $[0, T_1]$ provided that $\frac{1}{4}-4C^2 \|\w_0\|_{\LL^3(\mathds{R}^3}>0$ holds, which is true  due to (\ref{sec5-th1-lab8}). In particular, we have 
\begin{equation}\label{sec5-th1-lab10-bis}
\|\w\|_{L^5(0,T_1;\LL^5(\R^3))}\le \frac{1}{4 C}.
\end{equation}
In view of (\ref{sec5-th1-lab4}) and (\ref{sec5-th1-lab6}), estimates (\ref{sec5-th1-lab9}) and (\ref{sec5-th1-lab10}) are satisfied for $i=0$ by using (\ref{sec5-th1-lab10-bis}).

In general, for  $i\ge 1$, assume that (\ref{sec5-th1-lab9}) and (\ref{sec5-th1-lab10}) hold for $i-1$. Then if we argue as before,  we see that 
$$
0<K'= \frac{\frac{1}{2}-\sqrt{\frac{1}{4}-4C^2 \|\w(T_i)\|_{\LL^3(\mathds{R}^3)}}}{2C}.
$$
The induction hypothesis gives 
$$
\|\w(T_i)\|_{\LL^3(\R^3)}\le (2 C)^{i} \|\w_0\|_{\LL^3(\R^3)}< \frac{1}{8C (2C)^{M-i}} \le  \frac{1}{16 C^2}. 
$$
Thus, 
$$
\|\w\|_{L^5(T_i,T_{i+1};\LL^5(\R^3))}\le \frac{1}{4 C}.
$$
Applying this to (\ref{sec5-th1-lab4}) and (\ref{sec5-th1-lab6}), we obtain that estimates (\ref{sec5-th1-lab9}) and (\ref{sec5-th1-lab10}) hold for $i$.

To complete the proof, note, by (\ref{sec5-th1-lab10}), that 
$$
\|\w(t)\|_{\LL^3(\R^5))}\le (2 C)^M\|\w_0\|_{\LL^3(\R^3)}\quad\mbox{ for all }\quad t\in[0,T],
$$             
whence (\ref{sec5-th1-L3-stability}) holds.  
\end{proof}

The question that remains open is whether our particular solutions provided by Theorem \ref{Main}  can develop  singularities on $(T,\infty)$. Unfortunately, we are only able to give a partial answer to this question based on the following assumption. Let $\vv_0=\u_{0,\varepsilon}^{\tilde\lambda}+\w_{0,\varepsilon}^{\hat\lambda}$ be as in the proof of Theorem \ref{Main} and satisfy conditions (\ref{sec3-lm1-smallness_1}) and (\ref{sec3-lm1-smallness_2}). Then we suppose that, for each $\varepsilon>0$, there exists $\hat\lambda_0$ such that, for all $\hat\lambda\ge\hat\lambda_0$, it follows that    
\begin{equation}\label{A}
\tag{A} \left|\int_{0}^{\frac{T}{2}}\|\nabla\vv(s)\|^2_{\LL^2(\R^3)} \ds - \frac{1}{2}\|\vv_0\|^2_{\LL^2(\R^3)}\right|<\varepsilon.
\end{equation}
In other words, we shall look for initial data $\vv_0$ whose corresponding solution to (\ref{PDE-NS})--(\ref{IC}) has $\LL^2(\R^3)$-values of the vorticity, i.e. $\|\nabla\times\vv(t)\|_{\LL^2(\R^3)}=\|\nabla\vv(t)\|_{\LL^2(\R^3)}$, sufficiently high on $(0,T)$ such that assumption (\ref{A}) holds.

From our special choice of initial data $\vv_0=\u_{0,\varepsilon}^{\tilde\lambda}+\w_{0,\varepsilon}^{\hat\lambda}$, we can take $\tilde\lambda$ to tend to $\infty$ without increasing $\|\vv_0\|_{\LL^2(\Omega)}$, but $\|\nabla\vv_0\|_{\LL^2(\R^3)}$ does. Then we would expect that the vorticity does keep high via the vortex stretching mechanism for a certain period of time; and hence the kinetic energy would decay up to a certain threshold on $[0,T]$. 

\begin{theorem} Let $T>1$.  Assume that assumption (\ref{A}) holds. Then the solution $\vv(t)$ to (\ref{PDE-NS})--(\ref{IC}) provided by Theorem~\ref{Main} with $\vv_0=\u^{\tilde\lambda}_{0,\varepsilon}+\w_{0,\varepsilon}^{\hat\lambda}$ are smooth on $[0,\infty)$.
\end{theorem} 
\begin{proof}
From (\ref{Energy}), we find 
$$
\frac{1}{2}\|\vv(\frac{T}{2})\|^2_{\LL^2(\R^3)}+\int_{0}^{\frac{ T}{2}}\|\nabla\vv(s)\|^2_{\LL^2(\R^3)}\,\ds=
\frac{1}{2}\|\vv_0\|^2_{\LL^2(\R^3)}.
$$ 
In virtue of assumption (\ref{A}), we infer that $\|\vv(\frac{T}{2})\|_{\LL^2(\R^3)}<2\varepsilon$. Moreover, we have
$$
\frac{1}{2}\|\vv(t)\|^2_{\LL^2(\R^3)}+\int_{\frac{T}{2}}^t\|\nabla\vv(s)\|^2_{\LL^2(\R^3)}\,\ds=
\frac{1}{2}\|\vv (\frac{T}{2})\|^2_{\LL^2(\R^3)}.
$$ 
for all $t\in[\frac{T}{2}, T]$.

As in the proof of Lemma \ref{sec4-lm2}, we write 
$$
\|\vv(t)\|_{\LL^3(\R^3)}\le C \|\vv(t)\|_{\LL^2(\R^3)}^\frac{1}{2} \|\nabla\vv(t)\|^\frac{1}{2}_{\LL^2(\R^3)}.
$$
Taking the fourth power of both sides and integrating over $(\frac{T}{2}, T)$ yields 
$$
\int_{\frac{T}{2}}^{T}\|\vv(t)\|^4_{\LL^3(\R^3)}\dt\le \frac{C}{4}\|\vv(\frac{T}{2})\|^4_{\LL^2(\R^3)}. 
$$
Therefore,
$$
\frac{T}{2} \inf_{s\in[\frac{T}{2},T]} \|\vv(s)\|^4_{\LL^3(\R^3)}\le \frac{C}{4} \|\vv(\frac{T}{2})\|^4_{\LL^2(\R^3)}  
$$
and hence
$$
\inf_{s\in[\frac{T}{2},T]} \|\vv(s)\|_{\LL^3(\R^3)}\le C 2^{-\frac{1}{4}} T^{-\frac{1}{4}} \|\vv(\frac{T}{2})\|_{\LL^2(\R^3)}< C \|\vv(\frac{T}{2})\|_{\LL^2(\R^3)}.  
$$
Let us choose $\varepsilon<\frac{1}{8 C^2}$. Then we find that there exists $t^*\in(0, T]$ such that it follows that 
$$
\|\vv(t^*)\|_{\LL^3(\R^3)}<\frac{1}{4 C}.
$$ 
We are now allowed to apply Lemma \ref{sec4-lm1} to obtain that $\|\vv(t)\|_{\LL^3(\R^3)}\le \frac{1}{2C}$ for all $t\in[T,\infty)$. Moreover, we know that  $\|\vv(t)\|_{\LL^3(\R^3)}\le M_\vv$ for all $t\in[0,T]$ by Lemma~\ref{sec3-lm3}.  As a result of Theorem \ref{ESS}, the solution $\vv(t)$ is smooth globally in time.

\end{proof}

\end{document}